\documentclass[12pt]{article}

\usepackage{amsmath}
\usepackage{amssymb}
\usepackage{amsthm}
\usepackage{bm}
\usepackage{enumitem}
\usepackage{cite}

\usepackage{tikz}
\tikzset{every node/.style={draw,circle,inner sep=2pt}}
\usetikzlibrary{matrix}

\theoremstyle{plain}
\newtheorem{thm}{Theorem}[section]
\newtheorem{cor}[thm]{Corollary}
\newtheorem{lem}[thm]{Lemma}

\theoremstyle{definition}
\newtheorem{defn}[thm]{Definition}
\newtheorem{exam}[thm]{Example}
\newtheorem{rem}[thm]{Remark}
\newtheorem{obs}[thm]{Observation}

\theoremstyle{remark}
\numberwithin{equation}{section}

\newcommand{\calD}{\mathcal{D}}
\newcommand{\cof}{\operatorname{cof}}
\newcommand{\cofD}{{\textstyle\cof_{\calD}}}
\newcommand{\detD}{{\textstyle\det_{\calD}}}
\newcommand{\inertia}{\operatorname{inertia}}
\newcommand{\inertiaD}{\inertia_{\calD}}
\newcommand{\CP}[1]{\mathcal{CP}_{#1}}

\newcommand{\bb}{{\bf b}}
\newcommand{\be}{{\bf e}}
\newcommand{\bd}{{\bf d}}

\newcommand{\by}{{\bf y}}
\newcommand{\bphi}{\bm{\phi}}
\newcommand{\bzero}{{\bf 0}}
\newcommand{\bone}{{\bf 1}}
\newcommand{\trans}{^\top}
\newcommand{\floating}[1]{\smash{\raisebox{.5\normalbaselineskip}{#1}}}
\newcommand{\sign}{\operatorname{sign}}
\newcommand{\dist}{\operatorname{dist}}
\newcommand{\dunion}{\mathbin{\dot{\cup}}}

\begin{document}
\title{On the distance matrices of the CP graphs}
\author{Yen-Jen Cheng\footnotemark[2] \and Jephian C.-H. Lin\footnotemark[3]}
\date{\today}

\maketitle

\renewcommand{\thefootnote}{\fnsymbol{footnote}}
\footnotetext[2]{
    Department of Applied Mathematics, National Chiao Tung University, 1001 University Road, Hsinchu, Taiwan 30010
		}

\footnotetext[3]{
    Department of Applied Mathematics, National Sun Yat-sen University, Kaohsiung 80424, Taiwan
		}

\renewcommand{\thefootnote}{\arabic{footnote}}

\begin{abstract}
This paper introduces a new class of graphs, the CP graphs, and shows that their distance determinant and distance inertia are independent of their structures.  The CP graphs include the family of linear $2$-trees.  When a graph is attached with a CP graph, it is shown that the distance determinant and the distance inertia are also independent of the structure of the CP graph.  Applications to the addressing problem proposed by Graham and Pollak in 1971 are given.


\end{abstract}

\section{Introduction}
\label{sec:intro}

On a simple connected graph $G$, the distance between two vertices $i$ and $j$ is the length of the shortest path between them, denoted as $\dist_G(i,j)$.  The \emph{distance matrix} of connected graph $G$ is
\[\calD(G)=\begin{bmatrix}\dist_G(i,j)\end{bmatrix}.\]
Various properties of the distance matrix of a graph has been studied intensively; see \cite{AH14} for a survey and the references therein.

Suppose $A$ is an $n\times n$ symmetric matrix.  The \emph{inertia} $\inertia(A)$ of $A$ is the triple $(n_+,n_-,n_0)$, where $n_+$, $n_-$, and $n_0$ are the number of positive, negative, and zero eigenvalues of $A$, respectively.  The \emph{$i,j$-cofactor} of $A$ is $(-1)^{i+j}\det( A(i|j))$, where $A(i|j)$ is the matrix obtained from $A$ by removing the $i$-th row and the $j$-th column.  Let $\cof(A)$ be the sum of all cofactors.  That is,
\[\cof(A)=\sum_{i=1}^n\sum_{j=1}^n (-1)^{i+j} \det (A(i|j)).\]
When $A$ is invertible, $\cof (A)=\det(A)(\bone\trans A^{-1}\bone)$, where $\bone$ is the all-one vector.  For convenience, we write $\detD(G)=\det(\calD(G))$, $\inertiaD(G)=\inertia(\calD(G))$, and $\cofD(G)=\cof (\calD(G))$.

In 1971, Graham and Pollak \cite{GP71} proved that $\detD(T)=(-1)^{n-1}(n-1)2^{n-2}$ for any tree $T$ on $n$ vertices, so the distance determinant is independent of the structure of the tree.  (In 2006, a simple proof of this result was given in \cite{YY06}.)  Graham, Hoffman, and Hosoya \cite{GHH77} then gave a generalization by showing both $\detD(G)$ and $\cofD(G)$ are determined by $\detD(G_i)$ and $\cofD(G_i)$ for $i=1,\ldots,k$, where $G_i$'s are the blocks of $G$.  (The blocks of a graph is the maximal induced subgraphs of $G$ without a cut-vertex.)  This means $\detD(G)$ and $\cofD(G)$ are independent of how the blocks are attached to each other.  Several variants of the distance matrices were considered, such as the weighted distance matrix \cite{BKN05}, the $q$-analog and the $q$-exponential distance matrix \cite{BLP06,YY07}, and the determinant of these matrices of a tree are shown to be independent of the structure of the tree.  These results gave elegant formulas for various types the distance matrices of a tree, or graphs with cut-vertices.  A natural question is:  Can the distance determinant be a constant for other family of graphs with tree-like structure, or graphs without a cut-vertex?

\begin{figure}[h]
\begin{center}
\begin{tikzpicture}
\node (0) at (0,0) {};
\node (1) at (1,0) {};
\node (2) at (2,0) {};
\node (3) at (0.5,0.866) {};
\node (4) at (1.5,0.866) {};
\node (5) at (1,1.732) {};
\draw (0) -- (1) -- (2) -- (4) -- (5) -- (3) -- (0);
\draw (1) -- (4) -- (3) -- (1);
\node[draw=none,rectangle] at (1,-0.5) {$G_1$};
\end{tikzpicture}
\hfil
\begin{tikzpicture}
\node (0) at (0,0) {};
\node (1) at (1,0) {};
\node (2) at (2,0) {};
\node (3) at (0.5,-0.866) {};
\node (4) at (1.5,-0.866) {};
\node (5) at (2.5,-0.866) {};
\draw (0) -- (1) -- (2) -- (5) -- (4) -- (3) -- (0);
\draw (3) -- (1) -- (4) -- (2);
\node[draw=none,rectangle] at (1.25,-1.366) {$G_2$};
\end{tikzpicture}
\hfil
\begin{tikzpicture}
\node (0) at (0,0) {};
\node (1) at (1,0) {};
\node (2) at (2,0) {};
\node (3) at (0.5,-0.866) {};
\node (4) at (1.5,-0.866) {};
\node (5) at (1.5,0.866) {};
\draw (0) -- (1) -- (2) -- (4) -- (3) -- (0);
\draw (1) -- (4);
\draw (3) -- (1) -- (5) -- (2);
\node[draw=none,rectangle] at (1,-1.366) {$G_3$};
\end{tikzpicture}

\end{center}
\caption{Three $2$-trees}
\label{threektrees}
\end{figure}
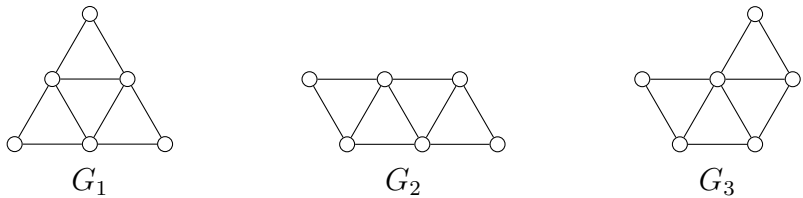

The family of $k$-trees is an immediate candidate to answer the question.  Figure~\ref{threektrees} shows three $2$-trees $G_1$, $G_2$, and $G_3$ of the same order.  By direct computation, $\detD(G_1)=-8$ and $\detD(G_2)=\detD(G_3)=-9$.  It seems giving a negative answer to the questions.  However, it also suggests that the family of linear $2$-tree is probably promising.

Indeed, Corollary~\ref{cor:lineartwotree} shows every linear $2$-tree on $n$ vertices have the same distance determinant.  Section~\ref{sec:twoCP} defines the $2$-clique paths, which is obtained by gluing the edges from several cliques into a path-like structure; the family of linear $2$-tree is a special case of the family of $2$-clique paths.  Theorem~\ref{thm:2cpdetinertia} shows the distance determinant only depends on the size of each clique.  In fact, the $2$-clique paths belong to a bigger family of graphs, the CP graphs; see Figure~\ref{G1G2} for some examples of the CP graphs.  Section~\ref{sec:CP} defines the CP graphs and shows that the distance determinant of the CP graphs only depends on the input parameters; the distance inertia and $\cofD(G)$ are also considered.  Section~\ref{sec:attaching} shows gluing the initial edges of two different CP graphs to any given connected graph will give two new graphs with the same distance determinant.

Finally, we turn our attention to the addressing problem.  Graham and Pollak \cite{GP71} proposed an \emph{addressing scheme} on a graph $G$ as an assignment of strings in $\{0,1,*\}$ of length $d$ to each vertex such that the distance between any pair of vertices is equal to the Hamming distance of their strings, ignoring the digits of $*$.  Let $N(G)$ be the minimum $d$ so that there is an addressing scheme on $G$.  In \cite{GP71}, it was shown that
\[N(G)\geq \max\{n_+,n_-\},\]
where $\inertiaD(G)=(n_+,n_-,n_0)$, and was conjectured that $N(G)\leq n-1$ for any graph of order $n$.
This conjecture was then proved by Winkler \cite{Winkler83} in 1983.  If $G$ is a tree or a block graph on $n\geq 2$ vertices, then it is known \cite{GP71,LLL15} that $n_-=N(G)=n-1$.  Section~\ref{sec:address} shows that $\inertiaD(G)=(1,n-1,0)$ and $N(G)=n-1$ for any graph $G$ whose blocks are $2$-clique paths, generalizing the known results.

For convenience, we use weighted graphs to record a matrix.  A \emph{weighted graph} is a simple graph whose vertices and edges are associated with weights.  The weight $w(i,j)$ of an edge $\{i,j\}$ is nonzero, and the weight $w(i)$ of a vertex can possibly be zero.  The \emph{weighted adjacency matrix} of a weighted graph is $\begin{bmatrix} a_{i,j}\end{bmatrix}$ with
\[\begin{cases}
a_{i,j}=w(i,j) & \text{if }i\neq j,\\
a_{i,j}=w(i) & \text{if }i=j.
\end{cases}\]
The notation $\{\be_i\}_{i=1}^n$ stands for the standard basis of $\mathbb{R}^n$.  Define $[n]=\{1,\ldots, n\}$ and $[a,b]=\{a,a+1,\ldots, b\}$.  When $a>b$, $[a,b]=\varnothing$.

\section{The CP graphs and their distance matrices}
\label{sec:CP}

A sequence of integers $q_1,\ldots,q_n$ ($n\geq 2$) is called \emph{non-leaping} if $q_1=0$, $q_2=1$, and $2\leq q_k\leq q_{k-1}+1$ for any $k=3,\ldots, n$.  (So $q_3=2$ if $n\geq 3$.)
For a given non-leaping sequence, a \emph{neighborhood sequence} is a sequence of sets $W_1,\ldots, W_n$ with the following properties:
\begin{enumerate}
\item $W_1=\varnothing$;
\item $W_2=\{1\}$;
\item for each $k\geq 3$, $|W_k|=q_k$ and $W_k=\{a_k\}\cup [b_k,k-1]$, where $b_k=k-q_k+1$ and $a_k\in W_{k-1}$ with $a_k<b_k$.
\end{enumerate}
We vacuously define $a_2=1$ and $b_2=2$.  Note that each $b_k$ is determined by the given non-leaping sequence, yet the neighborhood sequences may vary by the choices of $a_k$.  Therefore, each set $W_k$ contains $q_k$ elements, while $q_k-1$ of them are fixed with respect to the non-leaping sequence.

A neighborhood sequence gives the construction of a graph:  Start with vertex $1$.  For $k=2,\ldots, n$, add vertex $k$ and join it to every vertex in $W_k$.  In other words, a neighborhood sequence describes the ``backward'' neighborhoods of its graph.  Note that a non-leaping sequence may lead to more than one neighborhood sequence, depending on the choices of $a_k$, but the neighborhood sequence uniquely determines the graph.

\begin{obs}
\label{firstobs}
By the definition of a non-leaping sequence, $q_k\leq q_{k-1}+1$ and $b_k\geq b_{k-1}$.  Therefore,
$W_k\setminus\{k-1\}=\{a_k\}\cup [b_k,k-2] \subseteq W_{k-1}$
by the definition of a neighborhood sequence.  As a result, if $G$ is the graph constructed by a neighborhood sequence, then $W_k\cup\{k\}$ forms a clique in $G$ for each $k$.
\end{obs}

\begin{defn}
Given a non-leaping sequence $q_1,\ldots, q_n$, the family $\CP{q_1,\ldots,q_n}$ consists of the graphs of any possible neighborhood sequences of $q_1,\ldots,q_n$.  Graphs in $\CP{q_1,\ldots,q_n}$ for any non-leaping sequence are called CP graphs.
\end{defn}

Note that when $s=0,1,2,\ldots,2$ with $m$ copies of $2$, the family $\CP{s}$ is exactly the family of linear $2$-trees on $m+2$ vertices; see, e.g., \cite{param} for the definition of linear $2$-trees.

\begin{defn}
\label{def:reduced}
The \emph{reduced graph} of a non-leaping sequence $q_1,\ldots,q_n$ is a weighted graph on $n$ vertices whose edges are $\{1,2\}$ with weight $1$ and
\[\begin{cases}
\{b_{k-1},k\} & \text{with weight }1, \\
\{b_k,k\} & \text{with weight }-1, \\
\{k-1,k\} & \text{with weight }1
\end{cases}\]
for $k=3,\ldots, n$.
For each $k$, if any edges of the three edges above are the  same, then they merge as an edge and the weight is the sum of the weight of each edge; when the sum of the weights is zero, the edge degenerates as a nonedge.
Finally, the weight for vertices $1$ and $2$ are $0$ while the weight of all other vertices is $-2$.
\end{defn}

Here we elaborate all cases of Definition~\ref{def:reduced}.  First observe that $b_{k-1}\leq b_k\leq k-1$ for $k\geq 3$ by definition.  Then note that $b_{k-1}=b_k$ if and only if $q_k=q_{k-1}+1$; also, $b_k=k-1$ if and only if $q_k=2$.
Therefore, the case $b_{k-1}=b_k=k-1$ happens only when $k=3$; in this case, the three edges merge together as a single edge $\{2,3\}$ with weight $1-1+1=1$.  Now consider the cases when $k\geq 4$.  When $q_k=2$, the edges $\{b_k,k\}$ and $\{k-1,k\}$ cancel with each other since the weight is $-1+1=0$; similarly, when $q_k=q_{k-1}+1$, the edges $\{b_{k-1},k\}$ and $\{b_k,k\}$ cancel with each other.

\begin{rem}
\label{rem:reduced}
For any $k=3,\ldots, n$, if $q_k=2$, then the only neighbor of $k$ in $[1,k-1]$ is $b_{k-1}$; if $q_k=q_{k-1}+1$, then the only neighbor of $k$ in $[1,k-1]$ is $k-1$; otherwise,  each $k$ has three neighbors in $[1,k-1]$.
\end{rem}

\begin{figure}[h]
\begin{center}
\begin{tikzpicture}
\node[label={above:$1$}] (1) at (-1,0) {};
\node[label={above:$3$}] (3) at (0,0) {};
\node[label={above:$5$}] (5) at (1,0) {};
\node[label={above:$7$}] (7) at (2,0) {};
\node[label={below:$2$}] (2) at (240:1) {};
\node[label={below:$4$},xshift=1cm] (4) at (240:1) {};
\node[label={below:$6$},xshift=2cm] (6) at (240:1) {};
\node[label={below:$8$},xshift=3cm] (8) at (240:1) {};
\draw (1) -- (3) -- (5) -- (7) -- (8) -- (6) -- (4) -- (2) -- (1);
\draw (2) -- (3) -- (4) -- (5) -- (6) -- (7);
\draw (4) -- (7);
\draw (5) -- (8);
\node[rectangle,draw=none] at (0.75,-2) {$G_1$};
\end{tikzpicture}
\hfil
\begin{tikzpicture}
\node[label={below:$1$}] (1) at (0,0) {};
\foreach \v/\ang in {2/180,3/240,4/300,5/0,6/60,7/120,8/150}{
\node[label={\ang:$\v$}] (\v) at (\ang:1) {};
\draw (1) -- (\v);
}
\draw (1) -- (2) -- (3) -- (4) -- (5) -- (6) -- (7) -- (8) -- (1);
\draw (5) -- (7);
\draw (6) -- (8);
\node[rectangle,draw=none] at (0,-2) {$G_2$};
\end{tikzpicture}
\end{center}
\caption{Two graphs $G_1$ and $G_2$ in $\CP{0,1,2,2,2,2,3,3}$}
\label{G1G2}
\end{figure}
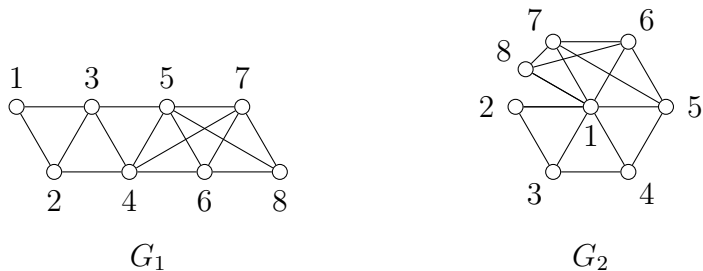

\begin{exam}
\label{exG1G2}
Let $s=0,1,2,2,2,2,3,3$ be a non-leaping sequence.  The two cases below are both neighborhood sequences of $s$.
\[\begin{array}{c|ccccccccc}
 & W_1 & W_2 & W_3 & W_4 & W_5 & W_6 & W_7 & W_8 \\
\hline
 G_1 & \varnothing & \{1\} & \{1,2\} & \{2,3\} & \{3,4\} & \{4,5\} & \{4,5,6\} & \{5,6,7\} \\
 G_2 & \varnothing & \{1\} & \{1,2\} & \{1,3\} & \{1,4\} & \{1,5\} & \{1,5,6\} & \{1,6,7\} \\
\end{array}
\]
The graphs $G_1$ and $G_2$ are shown in Figure~\ref{G1G2}, and both of them are in $\CP{s}$.

Notice that $a_k$ is always the minimum element in $W_k$ and $b_k$ is the next element after $a_k$ for any $k=3,\ldots,n$.  Therefore, the values of $b_k$ are invariants of $s$ as shown below.
\[\begin{array}{ccccccc}
b_2 & b_3 & b_4 & b_5 & b_6 & b_7 & b_8 \\
\hline
2 & 2 & 3 & 4 & 5 & 5 & 6
\end{array}
\]
Figure~\ref{redG} shows the reduced graph of $s$.  In general, each vertex $k$ with $k\geq 3$ is adjacent to three vertices in $[1,k-1]$; e.g., $k=8$.  However, these three edges might merge together.  If $q_k=2$, then $b_k=k-1$ and $\{b_k,k\}$ cancels $\{k-1,k\}$; e.g., $k=3,4,5,6$.  If $q_k=q_{k-1}+1$, then $b_{k-1}=b_k$ and $\{b_{k-1},k\}$ cancels with $\{b_k,k\}$; e.g., $k=7$.
\end{exam}

\begin{figure}[h]
\begin{center}
\begin{tikzpicture}[scale=1.5]
\node[label={above:$1$}] (1) at (-1,0) {$0$};
\node[label={above:$3$}] (3) at (0,0) {\tiny $-2$};
\node[label={above:$5$}] (5) at (1,0) {\tiny $-2$};
\node[label={above:$7$}] (7) at (2,0) {\tiny $-2$};
\node[label={below:$2$}] (2) at (240:1) {$0$};
\node[label={below:$4$},xshift=1cm] (4) at (240:1) {\tiny $-2$};
\node[label={below:$6$},xshift=2cm] (6) at (240:1) {\tiny $-2$};
\node[label={below:$8$},xshift=3cm] (8) at (240:1) {\tiny $-2$};
\draw (1) -- (2) -- (3) -- (5) -- (8) -- (7) -- (6) -- (4) -- (2);
\draw[dash dot] (6) -- (8);
\end{tikzpicture}
\end{center}
\caption{The reduced graph of the sequence $0,1,2,2,2,2,3,3$, where each solid edge has weight $1$ and each dashed edge has weight $-1$}
\label{redG}
\end{figure}
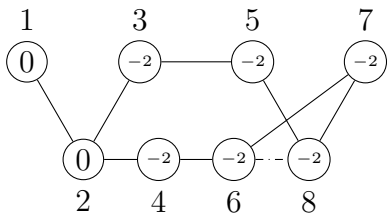

\begin{defn}
Let $s$ be a non-leaping sequence and $W_k$'s a corresponding neighborhood sequence.
Define the \emph{reducing matrix} $E$ as an $n\times n$ matrix whose $k$-th column is
\[\begin{cases}
\be_k & \text{if }k\in\{1,2\},\\
\be_k-\be_{a_k}-\be_{k-1}+\be_{a_{k-1}} & \text{if }k\geq 3. \\
\end{cases}\]
\end{defn}

\begin{exam}
\label{E1E2}
Let $G_1$ and $G_2$ be the graphs shown in Figure~\ref{G1G2}.  The reducing matrix depends on the choice of the neighborhood sequence.  For $i\in\{1,2\}$, the distance matrix $D_i$ and the corresponding reducing matrix $E_i$ of $G_i$ are shown below.
\[D_1=
\begin{bmatrix}
0 & 1 & 1 & 2 & 2 & 3 & 3 & 3 \\
1 & 0 & 1 & 1 & 2 & 2 & 2 & 3 \\
1 & 1 & 0 & 1 & 1 & 2 & 2 & 2 \\
2 & 1 & 1 & 0 & 1 & 1 & 1 & 2 \\
2 & 2 & 1 & 1 & 0 & 1 & 1 & 1 \\
3 & 2 & 2 & 1 & 1 & 0 & 1 & 1 \\
3 & 2 & 2 & 1 & 1 & 1 & 0 & 1 \\
3 & 3 & 2 & 2 & 1 & 1 & 1 & 0 \\
\end{bmatrix}, E_1=\begin{bmatrix}
1 & 0 & 0 & 1 & 0 & 0 & 0 & 0 \\
0 & 1 & -1 & -1 & 1 & 0 & 0 & 0 \\
0 & 0 & 1 & -1 & -1 & 1 & 0 & 0 \\
0 & 0 & 0 & 1 & -1 & -1 & 0 & 1 \\
0 & 0 & 0 & 0 & 1 & -1 & 0 & -1 \\
0 & 0 & 0 & 0 & 0 & 1 & -1 & 0 \\
0 & 0 & 0 & 0 & 0 & 0 & 1 & -1 \\
0 & 0 & 0 & 0 & 0 & 0 & 0 & 1 \\
\end{bmatrix}\]
\[D_2=\begin{bmatrix}
0 & 1 & 1 & 1 & 1 & 1 & 1 & 1 \\
1 & 0 & 1 & 2 & 2 & 2 & 2 & 2 \\
1 & 1 & 0 & 1 & 2 & 2 & 2 & 2 \\
1 & 2 & 1 & 0 & 1 & 2 & 2 & 2 \\
1 & 2 & 2 & 1 & 0 & 1 & 1 & 2 \\
1 & 2 & 2 & 2 & 1 & 0 & 1 & 1 \\
1 & 2 & 2 & 2 & 1 & 1 & 0 & 1 \\
1 & 2 & 2 & 2 & 2 & 1 & 1 & 0 \\
\end{bmatrix}, E_2=\begin{bmatrix}
1 & 0 & 0 & 0 & 0 & 0 & 0 & 0 \\
0 & 1 & -1 & 0 & 0 & 0 & 0 & 0 \\
0 & 0 & 1 & -1 & 0 & 0 & 0 & 0 \\
0 & 0 & 0 & 1 & -1 & 0 & 0 & 0 \\
0 & 0 & 0 & 0 & 1 & -1 & 0 & 0 \\
0 & 0 & 0 & 0 & 0 & 1 & -1 & 0 \\
0 & 0 & 0 & 0 & 0 & 0 & 1 & -1 \\
0 & 0 & 0 & 0 & 0 & 0 & 0 & 1 \\
\end{bmatrix}\]
Notice that the distance matrix and the reducing matrix depends on the choice of $a_k$.  However, by direct computation we will see that
\[E_1\trans D_1 E_1=E_2\trans D_2E_2=\begin{bmatrix}
0 & 1 & 0 & 0 & 0 & 0 & 0 & 0 \\
1 & 0 & 1 & 1 & 0 & 0 & 0 & 0 \\
0 & 1 & -2 & 0 & 1 & 0 & 0 & 0 \\
0 & 1 & 0 & -2 & 0 & 1 & 0 & 0 \\
0 & 0 & 1 & 0 & -2 & 0 & 0 & 1 \\
0 & 0 & 0 & 1 & 0 & -2 & 1 & -1 \\
0 & 0 & 0 & 0 & 0 & 1 & -2 & 1 \\
0 & 0 & 0 & 0 & 1 & -1 & 1 & -2 \\
\end{bmatrix},\]
which is the weighted adjacency matrix of the reduced graph as shown in Figure~\ref{redG}.
\end{exam}

Example~\ref{E1E2} illustrates the main result Theorem~\ref{finalthm}:  Given a non-leaping sequence $s$, for any graph $G\in\CP{s}$, its distance matrix $D$ and its reducing matrix $E$ always have $E\trans DE$ equal to the weighted adjacency matrix of the reduced graph.  Before showing Theorem~\ref{finalthm}, we need some lemmas.

\begin{lem}
\label{abclem}
Suppose $G$ is a graph constructed by a neighborhood sequence.  If $\{a,b\}\in E(G)$ with $a<b$, then $a$ is adjacent to any $c$ with $a<c<b$.
\end{lem}
\begin{proof}
By Observation~\ref{firstobs}, $W_b\setminus\{b-1\}\subseteq W_{b-1}$.  Therefore, $\{a,b\}\in E(G)$ implies $\{a,b-1\}\in E(G)$ if $a<b-1$.  Inductively, $a$ is adjacent to any $c$ with $a<c<b$.
\end{proof}

\begin{cor}
\label{abccor}
Suppose $G$ is a graph constructed by a neighborhood sequence.  For any two vertices $a<b$, each vertex $k$ on any shortest path  $a$ to $b$ has $k\leq b$.
\end{cor}
\begin{proof}
Suppose $a=v_1,\ldots,v_d=b$ be a shortest path from $a$ to $b$.  Suppose $v_i<b<v_{i+1}$ for some $i$.  Then $v_i$ is adjacent to $b$ by Lemma~\ref{abclem}, giving a shorter path from $a$ to $b$.
\end{proof}

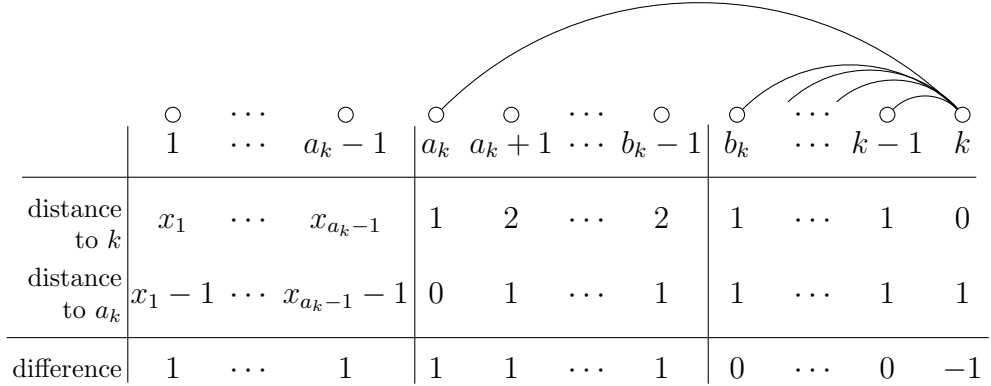
\begin{figure}[h]
\begin{center}
\begin{tikzpicture}
\node [rectangle, draw=none, matrix, matrix of math nodes, ampersand replacement=\&,
column sep={1cm,between origins},
row sep={1cm,between origins},
row 1 column 2/.style={nodes={rectangle,draw=none,anchor=center}},
row 1 column 6/.style={nodes={rectangle,draw=none,anchor=center}},
row 1 column 9/.style={nodes={rectangle,draw=none,anchor=center}},
row 2/.style={nodes={rectangle, draw=none}},
row 3/.style={nodes={rectangle, draw=none}},
row 4/.style={nodes={rectangle, draw=none}},
row 5/.style={nodes={rectangle, draw=none}}
] (A) at (0,0)
{ {} \& \cdots \&[0.3cm] {} \&[0.2cm] {} \& {} \& \cdots \& {} \& {} \& \cdots \& {} \& {} \\[-0.5cm]
 1 \& \cdots \& a_k-1 \& a_k \& a_k+1 \& \cdots \& b_k-1 \& b_k \& \cdots \& k-1 \& k \\
 x_1 \& \cdots \& x_{a_k-1} \& 1 \& 2 \& \cdots \& 2 \& 1 \& \cdots \& 1 \& 0 \\
 x_1-1 \& \cdots \& x_{a_k-1}-1 \& 0 \& 1 \& \cdots \& 1 \& 1 \& \cdots \& 1 \& 1 \\
  1 \& \cdots \& 1 \& 1 \& 1 \& \cdots \& 1 \& 0 \& \cdots \& 0 \& -1 \\
};
\draw (A-1-11) to[bend right=45] (A-1-10);
\draw (A-1-11) to[bend right=45] (A-1-9.north west);
\draw (A-1-11) to[bend right=45] (A-1-9.north east);
\draw (A-1-11) to[bend right=45] (A-1-8);
\draw (A-1-11) to[bend right=45] (A-1-4);

\node[rectangle,draw=none,xshift=-0.6cm,left,align=right,font=\footnotesize] (l3) at (A-3-1) {distance\\to $k$};
\node[rectangle,draw=none,xshift=-0.6cm,left,align=right,font=\footnotesize] (l4) at (A-4-1) {distance\\to $a_k$};
\node[rectangle,draw=none,xshift=-0.6cm,left,font=\footnotesize] (l5) at (A-5-1) {difference};

\draw ([yshift=0.2cm] l3.north west) -- ([yshift=0.2cm] l3.north west -| A-3-11.east);
\draw ([yshift=0.2cm] l5.north west) -- ([yshift=0.2cm] l5.north west -| A-5-11.east);
\draw (l5.south east) -- (l5.south east |- A-2-1.north);
\draw (A-4-3.east |- A-5-3.south) -- (A-4-3.east |- A-2-3.north);
\draw (A-2-7.east |- A-5-7.south) -- (A-2-7.south east |- A-2-7.north);
\end{tikzpicture}
\end{center}
\caption{An illustration of Lemma~\ref{difflem}}
\label{illlem}
\end{figure}

\begin{lem}
\label{difflem}
Let $s$ be a non-leaping sequence and $G\in\CP{s}$ with distance matrix $\calD$.  For $k\geq 2$, let $\bd=\calD(\be_k-\be_{a_k})$.  Then the $h$-th entry of $\bd$ with $h\leq k$ is
\[\bd_h=\begin{cases}
1 & \text{if }h<b_k,\\
0 & \text{if }b_k\leq h<k,\\
-1 & \text{if }h=k.
\end{cases}\]
\end{lem}
\begin{proof}
It is sufficient to verify the distances given in Figure~\ref{illlem}.  By definition, $a_2=1$ and $b_2=2$, so the result holds for $k=2$.  For the following, we assume $k\geq 3$.

Suppose $h$ is a vertex with $h<a_k$.  Pick a shortest path from $h$ to $k$, and let $b$ be the last vertex on the path before reaching $k$ and $a$ the previous vertex of $b$.  By Corollary~\ref{abccor}, $b\in W_k\subseteq [a_k,k-1]$.  If $b=a_k$, then we found a shortest path from $h$ to $k$ through $a_k$.  Suppose $b>a_k$.  If $a<a_k$, then $a$ is adjacent to $a_k$ by Lemma~\ref{abclem} since $\{a,b\}\in E(G)$ and $a<a_k<b$; if $a>a_k$, then $a$ is adjacent to $a_k$ by Lemma~\ref{abclem} since $\{a_k,k\}\in E(G)$ and $a_k<a<k$.  In either cases, we can replace $b$ by $a_k$, and then $a_k$ is adjacent to $k$.  It follows that every vertex $h$ with $h<a_k$ has a shortest path from $h$ to $k$ through $a_k$, so
\[\dist_G(h,k)=\dist_G(h,a_k)+1.\]

Suppose $h$ is a vertex with $a_k<h<b_k$.  Then $a_k$ is adjacent to $h$ by Lemma~\ref{abclem} since $\{a_k,k\}\in E(G)$ and $a_k<h<k$.  Therefore, $\dist_G(h,k)=2$ and $\dist_G(h,a_k)=1$.

Other cases are straightforward.
\end{proof}

\begin{cor}
\label{diffcor}
Let $s$ be a non-leaping sequence and $G\in\CP{s}$ with distance matrix $\calD$.  For $k\geq 3$, let $\bd=\calD(\be_{k-1}-\be_{a_{k-1}})$.  Then the $h$-th entry of $\bd$ with $h\leq k$ is
\[\bd_h=\begin{cases}
1 & \text{if }h<b_{k-1},\\
0 & \text{if }b_{k-1}\leq h<k-1,\\
-1 & \text{if }h=k-1, \\
0 & \text{if }h=k, a_{k-1}=a_k,\\
-1 & \text{if }h=k, a_{k-1}<a_k.\\
\end{cases}\]
\end{cor}
\begin{proof}
The cases with $h\leq k-1$ follow from Lemma~\ref{difflem}.  If $a_{k-1}=a_k$, then
\[\dist_G(k,k-1)=1=\dist_G(k,a_{k-1})\]
and $\bd_{k}=0$.  If $a_{k-1}<a_k$, then
\[\dist_G(k,k-1)=1 \text{ and } \dist_G(k,a_{k-1})=2,\]
so $\bd_k=-1$.
\end{proof}

\begin{cor}
\label{finalcor}
Let $s$ be a non-leaping sequence and $G\in\CP{s}$ with distance matrix $\calD$.  For $k\geq 3$, let $\bb=\calD(\be_k-\be_{a_k}-\be_{k-1}+\be_{a_{k-1}})$.  Then the $h$-th entry of $\bb$ with $h\leq k$ is
\[\bb_h=\begin{cases}
0 & \text{if }h<b_{k-1},\\
1 & \text{if }b_{k-1}\leq h<b_k,\\
0 & \text{if }b_k\leq h<k-1, \\
1 & \text{if }h=k-1, \\
-1 & \text{if }h=k, a_{k-1}=a_k,\\
0 & \text{if }h=k, a_{k-1}<a_k.\\
\end{cases}\]
\end{cor}
\begin{proof}
Let $\bd^{(k)}=\calD(\be_k-\be_{a_k})$ and $\bd^{(k-1)}=\calD(\be_{k-1}-\be_{a_{k-1}})$.  By definition,
\[\bb=\calD(\be_k-\be_{a_k})-\calD(\be_{k-1}-\be_{a_{k-1}})=\bd^{(k)}-\bd^{(k-1)},\]
so the result follows directly from Lemma~\ref{difflem} and Corollary~\ref{diffcor}.    For convenience, we align the results from Lemma~\ref{difflem} and Corollary~\ref{diffcor} together.
\[\bd^{(k)}_h=\begin{cases}
1 & \text{if }h<b_{k-1},\\
1 & \text{if }b_{k-1}\leq h<b_k,\\
0 & \text{if }b_k\leq h<k-1, \\
0 & \text{if }h=k-1, \\
-1 & \text{if }h=k, a_{k-1}=a_k,\\
-1 & \text{if }h=k, a_{k-1}<a_k.\\
\end{cases}
\qquad
\bd^{(k-1)}_h=\begin{cases}
1 & \text{if }h<b_{k-1},\\
0 & \text{if }b_{k-1}\leq h<b_k,\\
0 & \text{if }b_k\leq h<k-1, \\
-1 & \text{if }h=k-1, \\
0 & \text{if }h=k, a_{k-1}=a_k,\\
-1 & \text{if }h=k, a_{k-1}<a_k.\\
\end{cases}\]
By taking the differences of the corresponding terms, this completes the proof.
\end{proof}

\begin{thm}
\label{finalthm}
Let $s$ be a non-leaping sequence with the reduced graph $H$.  For any neighborhood sequence of $s$ and the corresponding graph $G$ with distance matrix $\calD$, the matrix $E\trans\mathcal{D}E$ is the weighted adjacency matrix of $H$, where $E$ is the reducing matrix that depends on the neighborhood sequence.
\end{thm}
\begin{proof}
Let $R=E\trans DE=\begin{bmatrix}r_{i,j}\end{bmatrix}$.  Since $R$ is a symmetric matrix, it is sufficient to show that $R$ is the same as the weighted adjacency matrix of $H$ for the $i,j$-entry with $i\leq j$.

Since the principal submatrix of $R$ on the first two columns and rows is $\begin{bmatrix}0&1\\1&0\end{bmatrix}$, which agrees with $H$, we assume $j\geq 3$.

Let $\bphi_{i,j}=\sum_{h=i}^j \be_h$, which is defined as a zero vector when $j<i$. By Corollary~\ref{finalcor},
\[\calD E\be_j=\bphi_{b_{j-1},b_j-1}+\be_{j-1}-s\be_j+\by,\]
where $s=1,0$ depending on $a_{j-1}=a_j$ or not and $\by$ is a vector that vanish on all entries from $1$ to $j$.

First, we examine the diagonal entries.  For $i=j$ and $i,j\geq 3$,
\[\begin{aligned}
r_{j,j} &= \be_j\trans E\trans\calD E\be_j\\
 &= (\be_j-\be_{a_j}-\be_{j-1}+\be_{a_{j-1}})\trans(\bphi_{b_{j-1},b_j-1}+\be_{j-1}-s\be_j+\by) \\
 &= (\be_j-\be_{a_j}-\be_{j-1})\trans(\bphi_{b_{j-1},b_j-1}+\be_{j-1}-s\be_j) \\
 &=
 \begin{cases}
 (\be_j-\be_{j-1})\trans(\bphi_{b_{j-1},b_j-1}+\be_{j-1}-\be_j) & \text{if }a_{j-1}=a_j \\
 (\be_j-\be_{a_j}-\be_{j-1})\trans(\bphi_{b_{j-1},b_j-1}+\be_{j-1}) & \text{if }a_{j-1}<a_j\\
 \end{cases} \\
 &=-2. \\
\end{aligned}\]
Here we use the fact that $a_{j-1}=a_j$ implies $a_j<b_{j-1}$, and $a_{j-1}<a_j$ implies $b_{j-1}\leq a_j\leq b_j-1$.

If $j>i\geq 3$, then $a_i,a_{i-1}<b_{j-1}$.  Thus,
\[\begin{aligned}
r_{i,j} &= \be_i\trans E\trans\calD E\be_j\\
 &= (\be_i-\be_{a_i}-\be_{i-1}+\be_{a_{i-1}})\trans(\bphi_{b_{j-1},b_j-1}+\be_{j-1}-s\be_j+\by) \\
 &= (\be_i-\be_{i-1})\trans(\bphi_{b_{j-1},b_j-1}+\be_{j-1}) \\
 &= \be_i\trans\be_{j-1}+\be_i\trans\bphi_{b_{j-1},b_j-1}-\be_{i-1}\trans\bphi_{b_{j-1},b_j-1} \\
 &= \be_i\trans\be_{j-1}+\be_i\trans\bphi_{b_{j-1},b_j-1}-\be_{i}\trans\bphi_{b_{j-1}+1,b_j}. \\
\end{aligned}\]
In fact, this formula holds also for $i=1,2$ and $j\geq 3$.  When $i=1,2$, \[\begin{aligned}
r_{i,j} &= \be_i\trans E\trans\calD E\be_j\\
 &= \be_i\trans(\bphi_{b_{j-1},b_j-1}+\be_{j-1}-s\be_j+\by) \\
 &= \be_i\trans(\bphi_{b_{j-1},b_j-1}+\be_{j-1}) \\
 &= \be_i\trans\be_{j-1}+\be_i\trans\bphi_{b_{j-1},b_j-1}. \\
\end{aligned}\]
The two formulas agree since $b_{j-1}\geq 2$ for any $j$ and $\be_{i}\trans\bphi_{b_{j-1}+1,b_j+1}$ vanishes when $i=1,2$.  Therefore, for $j\geq 3$ and $i< j$,
\[r_{i,j}=\be_i\trans(\be_{j-1}+\bphi_{b_{j-1},b_j-1}-\bphi_{b_{j-1}+1,b_j}).\]
In other words,
\[\begin{aligned}
E\trans\calD E\be_j &=\be_{j-1}+\bphi_{b_{j-1},b_j-1}-\bphi_{b_{j-1}+1,b_j} + \by'\\
 &= \be_{j-1}-\be_{b_j}+\be_{b_{j-1}}+\by'.
\end{aligned},\]
where $\by'$ is a vector that vanishes on entries from $1$ to $j-1$.  Therefore, $R$ is the weighted adjacency matrix of $H$.
\end{proof}

\begin{cor}
\label{cor:detinertia}
Let $s$ be a non-leaping sequence with the reduced graph $H$.  Then for any graph $G\in \CP{s}$, both $\detD(G)$ and $\inertiaD(G)$ are uniquely determined by $s$.  Indeed, $\detD(G)=\det(A)$ and $\inertiaD(G)=\inertia(A)$ for any $G\in\CP{s}$, where $A$ is the weighted adjacency matrix of $H$.
\end{cor}
\begin{proof}
Let $\calD$ be the distance matrix of $G$ and $A$ the weighted adjacency matrix of $H$.  By Theorem~\ref{finalthm}, $A=E\trans\calD E$, where $E$ is the reducing matrix corresponding to $G$.  Since $E$ is an upper triangular matrix with diagonal entries equal to $1$, $\det(\calD)=\det(A)$.  Since $\calD$ is congruent to $A$, the inertia of $\calD$ is the same as the inertia of $A$.
\end{proof}

Let $J_k$ be the $k\times k$ all-one matrix.  Let $J_{2,n}$ be the matrix obtained by embedding $J_2$ to the top-left corner of the $n\times n$ zero matrix.

\begin{cor}
\label{cor:cof}
Let $s$ be a non-leaping sequence of length $n$ with the reduced graph $H$.  Then for any graph $G\in \CP{s}$, the value of $\cofD(G)$ is uniquely determined by $s$.  Indeed, $\cofD(G)=\det(A+J_{2,n})-\det(A)$ for any $G\in\CP{s}$, where $A$ is the weighted adjacency matrix of $H$.
\end{cor}
\begin{proof}
Let $\calD$ be the distance matrix of $G$ and $A$ the weighted adjacency matrix of $H$.  By Theorem~\ref{finalthm}, $A=E\trans\calD E$, where $E$ is the reducing matrix corresponding to $G$.  Notice that $E$ has column sums zero except for the first and the second columns.  Therefore,
\[E\trans(\calD+J)E=A+J_{2,n}.\]
By \cite[Lemma 9.3]{BapatGM14},
\[\begin{aligned}
\cofD(G) &= \det(\calD+J_n)-\det(\calD) \\
 &= \det(E\trans(\calD+J_n)E)-\det(A) \\
 &= \det(A+J_{2,n})-\det(A).
\end{aligned}\]
Therefore, the value $\cofD(G)$ is determined by $s$.
\end{proof}

\section{Attaching the CP graphs}
\label{sec:attaching}

Let $q_1,\ldots,q_n$ be a non-leaping sequence.
For any connected graph $G_0$ with an edge $e=\{v_1,v_2\}$ and for any $G\in\CP{q_1,\ldots,q_n}$, define $G_0\oplus_eG$ as the graph obtained from $G_0\dunion G$ by identifying the edges $\{v_1,v_2\}\in E(G_0)$ and $\{1,2\}\in E(G)$ with $v_1$ to $1$ and $v_2$ to $2$.

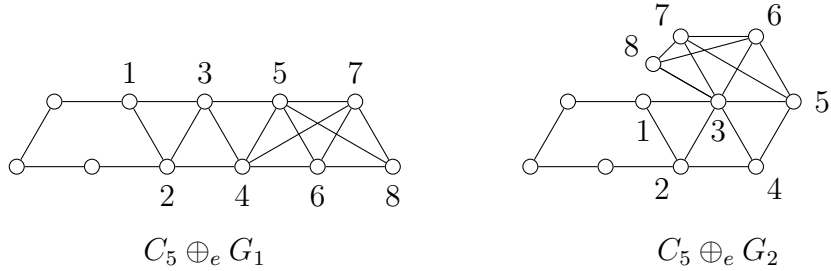
\begin{figure}[h]
\begin{center}
\begin{tikzpicture}
\node[label={above:$1$}] (1) at (-1,0) {};
\node[label={above:$3$}] (3) at (0,0) {};
\node[label={above:$5$}] (5) at (1,0) {};
\node[label={above:$7$}] (7) at (2,0) {};
\node[label={below:$2$}] (2) at (240:1) {};
\node[label={below:$4$},xshift=1cm] (4) at (240:1) {};
\node[label={below:$6$},xshift=2cm] (6) at (240:1) {};
\node[label={below:$8$},xshift=3cm] (8) at (240:1) {};
\node[label={above:}] (-1) at (-2,0) {};
\node[label={above:},xshift=-2cm] (-2) at (240:1) {};
\node[label={above:},xshift=-1cm] (-3) at (240:1) {};
\draw (1) -- (3) -- (5) -- (7) -- (8) -- (6) -- (4) -- (2) -- (1);
\draw (2) -- (3) -- (4) -- (5) -- (6) -- (7);
\draw (4) -- (7);
\draw (5) -- (8);
\draw (1) -- (-1) -- (-2) -- (-3) -- (2);
\node[rectangle,draw=none] at (0,-2) {$C_5\oplus_eG_1$};
\end{tikzpicture}
\hfil
\begin{tikzpicture}
\node[label={below:$3$}] (3) at (0,0) {};
\foreach \v/\ang in {2/240,4/300,5/0,6/60,7/120,8/150}{
\node[label={\ang:$\v$}] (\v) at (\ang:1) {};
\draw (3) -- (\v);
}
\node[label={below:$1$}] (1) at (180:1) {};
\node[label={above:}] (-1) at (-2,0) {};
\node[label={above:},xshift=-2cm] (-2) at (240:1) {};
\node[label={above:},xshift=-1cm] (-3) at (240:1) {};
\draw (3) -- (1) -- (2) -- (4) -- (5) -- (6) -- (7) -- (8) -- (3);
\draw (5) -- (7);
\draw (6) -- (8);
\draw (1) -- (-1) -- (-2) -- (-3) -- (2);
\node[rectangle,draw=none] at (0,-2) {$C_5\oplus_eG_2$};
\end{tikzpicture}
\end{center}
\caption{Two graphs $C_5\oplus_e G_1$ and $C_5\oplus_e G_2$ with $G_1,G_2\in\CP{0,1,2,2,2,2,3,3}$}
\label{C5G1G2}
\end{figure}

\begin{exam}
Let $C_5$ be the five cycle and $G_1,G_2$ the two graphs shown in Figure~\ref{G1G2}.  Let $e$ be an edge on $C_5$.  Then the two graphs $C_5\oplus_e G_1$ and $C_5\oplus_e G_2$ are shown in Figure~\ref{C5G1G2}.  Theorem~\ref{thm:attach} will show that $\detD(C_5\oplus_e G_1)=\detD(C_5\oplus_e G_2)$.
\end{exam}

\begin{obs}
\label{attachobs}
If $G_0$ is a connected graph with an edge $e=\{v_1,v_2\}$ and $G\in\CP{q_1,\ldots,q_n}$, then
\begin{enumerate}[label={\rm (\roman*)}]
\item $\dist_{G_0\oplus_eG}(x,y)=\dist_{G_0}(x,y)$ for $x,y\in V(G_0)$,
\item $\dist_{G_0\oplus_eG}(i,j)=\dist_{G}(i,j)$ for $i,j\in V(G)$,
\item for any $x\in V(G_0)$, $\dist_{G_0\oplus_e G}(x,3)$ is independent of the choice of $G$.
\end{enumerate}
\end{obs}

\begin{lem}
\label{difflem2}
Let $G_0$ be a connected graph with an edge $e=\{v_1,v_2\}$ and $G\in\CP{q_1,\ldots,q_n}$.  For $x\in V(G_0)\setminus\{v_1,v_2\}$ and $k\in V(G)\setminus\{1,2\}$,
$$\dist_{G_0\oplus_eG}(x,k)=\dist_{G_0\oplus_eG}(x,a_k)+1$$
if $k$ is not adjacent to $1$ and $2$ simultaneously.
\end{lem}
\begin{proof}
Pick a shortest path from $x$ to $k$.
Let $b$ be the last vertex on the path before reaching $k$ and $a$ the previous vertex.  Since $k\geq 3$, $b\in W_k\subseteq [a_k,k-1]$.  If $b=a_k$, then we found a path from $a$ to $b$ through $a_k$, so assume $b>a_k$.  If $a\notin V(G)$, then $b\in\{1,2\}$; by the assumption that $k$ is not adjacent to both $1$ and $2$, the only neighbor of $k$ in $\{1,2\}$ is $a_k=b$, a contradiction.
Suppose $a\in V(G)$ and $b>a_k$.  Either $a<a_k<b$ with $\{a,b\}\in E(G)$ or $a_k<a<k$ with $\{a_k,k\}\in E(G)$.  By Lemma~\ref{abclem}, $a$ is adjacent to $a_k$, and we may replace the vertex $b$ by $a_k$ and then $a_k$ is adjacent to $k$.  Therefore, every vertex $x\in V(G_0)\setminus\{v_1,v_2\}$ has a shortest path from $x$ to $k$ through $a_k$, so
\[\dist_{G_0\oplus_eG}(x,k)=\dist_{G_0\oplus_eG}(x,a_k)+1.\]
This completes the proof.
\end{proof}

\begin{lem}
\label{difflem3}
Let $G_0$ be a connected graph with an edge $e=\{v_1,v_2\}$ and $G\in\CP{q_1,\ldots,q_n}$.  For $x\in V(G_0)\setminus\{v_1,v_2\}$ and $k\in V(G)\setminus\{1,2\}$,
\[\dist_{G_0\oplus_eG}(x,k)=\dist_{G_0\oplus_eG}(x,k-1)\]
if $k\geq 4$ is adjacent to $1$ and $2$.
\end{lem}
\begin{proof}
Since $k$ is adjacent to $1$, $a_3=\cdots =a_k=1$.  With this assumption, $b_3=\cdots=b_k=2$ since $k$ is adjacent to $2$.  Therefore, $\{1,2,\ldots,k\}$ forms a clique from the definition of CP graphs, and
\[\begin{aligned}
 &\mathrel{\phantom{=}} \dist_{G_0\oplus_eG}(x,k)=\dist_{G_0\oplus_eG}(x,k-1) \\
 &=\min\{\dist_{G_0\oplus_eG}(x,1),\dist_{G_0\oplus_eG}(x,2)\}+1
\end{aligned}\]
whenever $k-1\geq 3$.
\end{proof}

\begin{thm}
\label{thm:attach}
Let $s=q_1,\ldots,q_n$ be a non-leaping sequence.
If $G_0$ is a connected graph with an edge $e=\{v_1,v_2\}$ and $G\in\CP{s}$, then $\detD(G_0\oplus_eG)$ is independent of the choice of $G$.
\end{thm}
\begin{proof}
By Observation~\ref{attachobs}, we may write $\calD(G_0\oplus_eG)$ as the form
\[
\begin{bmatrix}
\calD(G_0-\{v_1,v_2\}) & A_{12} \\
A_{12}\trans  &  \calD(G)
\end{bmatrix}.
\]
Let
\[
F=\begin{bmatrix}
I_{|V(G_0)|-2} & O \\
O & E
\end{bmatrix},
\]
where $E$ is the reducing matrix corresponding to $G$. Then
\[
F\trans\calD(G_0\oplus_eG)F=
\begin{bmatrix}
\calD(G_0-\{v_1,v_2\}) & A_{12}E \\
E\trans A_{12}\trans & E\trans\calD(G)E
\end{bmatrix}.
\]
We know that $\calD(G_0-\{v_1,v_2\})$ is independent of the choice of $G$, and so is $E^T\calD(G)E$ by Theorem \ref{finalthm}. It remains to show that $A_{12}E$ is independent of the choice of $G$.
Let $t\geq 4$ be the first vertex such that $q_t\ne t-1$. Then $t$ is not adjacent to $1$ and $2$ simultaneously, so is each $k$ for $k\geq t$ since $W_k-\{k-1\}\subseteq W_{k-1}$.
By Lemma \ref{difflem2},
\[\begin{aligned}
A_{12}E\be_k & =A_{12}(\be_k-\be_{a_k}-\be_{k-1}+\be_{a_{k-1}}) \\
 &=A_{12}(\be_k-\be_{a_k})-A_{12}(\be_{k-1}-\be_{a_{k-1}}) =\bone-\bone=\bzero
\end{aligned}\]
 if $k\geq t+1$.
When $k=t$,
\[\begin{aligned}
A_{12}E\be_k & =A_{12}(\be_k-\be_{a_k}-\be_{k-1}+\be_{a_{k-1}}) \\
 &=A_{12}(\be_k-\be_{a_k})-A_{12}(\be_{k-1}-\be_{a_{k-1}}) \\
 &=\bone- A_{12}(\be_{k-1}-\be_{a_{k-1}}) \\
 &=\bone- A_{12}(\be_{3}-\be_{1}),\
\end{aligned}\]
since $A_{12}\be_{k-1}=A_{12}\be_3$ by Lemma~\ref{difflem3} and $a_{k-1}=1$.

When $k\in [4,t-1]$, $a_k=a_{k-1}=1$ and $k$ is adjacent to both $1$ and $2$, so
\[\begin{aligned}
A_{12}E\be_k & =A_{12}(\be_k-\be_{a_k}-\be_{k-1}+\be_{a_{k-1}}) \\
 &=A_{12}(\be_k-\be_{k-1})=\bzero
\end{aligned}\]
by Lemma~\ref{difflem3}.  For $k=3$, $A_{12}E\be_k=A_{12}(\be_3-\be_2)$.  Therefore,
\[A_{12}E=
\begin{bmatrix}
A_{12}\be_1 & A_{12}\be_2 & A_{12}(\be_3-\be_2) & O & \bone-A_{12}(\be_{3}-\be_{1}) & O
\end{bmatrix},\]
where the first $O$ is a zero matrix with $t-4$ columns and the second $O$ has $|V(G)|-t$ columns.
Note that $A_{12}\be_1$ and $A_{12}\be_2$ are determined by $\calD(G_0)$, which is independent of the choice of $G$.  Also, $A_{12}\be_3$ is independent of the choice of $G$ by Observation~\ref{attachobs}.
And $t$ is determined only by $s$.
In conclusion, $\detD(G_0\oplus_e G)$ is independent of the choice of $G\in\CP{s}$.
\end{proof}

Let $e_1,\ldots, e_k$ be edges of a connected graph $G_0$.  Let $G_1,\ldots,G_k$ be CP graphs (not necessarily from the same non-leaping sequence).  Define $G_0\oplus_{e_1}G_1\oplus_{e_2}\cdots\oplus_{e_k}G_k$ as  $((G_0\oplus_{e_1}G_1)\oplus_{e_2}\cdots)\oplus_{e_k}G_k$. Note that
\[G_0\oplus_{e_1}G_1\oplus_{e_2}\cdots\oplus_{e_k}G_k=G_0\oplus_{e_{\pi_1}}G_{\pi_1}\oplus_{e_{\pi_2}}\cdots\oplus_{e_{\pi_k}}G_{\pi_k}\]
for any permutation $\pi=(\pi_1,\pi_2,\ldots,\pi_k)$ of $(1,2,\ldots,k)$.

\begin{cor}
\label{cor:attach}
Let $G_0$ be a connected graph.  For $1\leq i\leq k$, let $s^{(i)}$ be a non-leaping sequence, $e_i=\{v_1^{(i)},v_2^{(i)}\}$ an edge of $G_0$, and $G_i\in\CP{s^{(i)}}$.
Then $\detD(G_0\oplus_{e_1}G_1\oplus_{e_2}\cdots\oplus_{e_k}G_k)$ is independent of the choice of $G_1,\ldots,G_k$.
\end{cor}
\begin{proof}
Let $G_i,G'_i\in\CP{s^{(i)}}$ for $1\leq i\leq k$. Then
\[
\begin{aligned}
& \mathrel{\phantom{=}}\detD((G_0\oplus_{e_2}G_2\oplus_{e_3}\cdots\oplus_{e_k}G_k)\oplus_{e_1}G_1) \\
 & =\detD((G_0\oplus_{e_2}G_2\oplus_{e_3}\cdots\oplus_{e_k}G_k)\oplus_{e_1}G'_1)
\end{aligned}
\]
by Theorem \ref{thm:attach}. So
\[
\begin{aligned}
& \mathrel{\phantom{=}}\detD(G_0\oplus_{e_1}G_1\oplus_{e_2}\cdots\oplus_{e_k}G_k) \\
 &=
\detD((G_0\oplus_{e_2}G_2\oplus_{e_3}\cdots\oplus_{e_k}G_k)\oplus_{e_1}G_1) \\
 &= \detD((G_0\oplus_{e_2}G_2\oplus_{e_3}\cdots\oplus_{e_k}G_k)\oplus_{e_1}G'_1) \\
 &=
\detD(G_0\oplus_{e_1}G'_1\oplus_{e_2}G_2\oplus\cdots\oplus_{e_k}G_k).
\end{aligned}
\]
Inductively,
\[\begin{aligned}
& \mathrel{\phantom{=}}\detD(G_0\oplus_{e_1}G_1\oplus_{e_2}G_2\oplus\cdots\oplus_{e_k}G_k) \\
 &=
\detD(G_0\oplus_{e_1}G'_1\oplus_{e_2}G_2\oplus\cdots\oplus_{e_k}G_k) \\
 &=
\detD(G_0\oplus_{e_1}G'_1\oplus_{e_2}G'_2\oplus\cdots\oplus_{e_k}G_k) \\
 &= \cdots \\
 &=
\detD(G_0\oplus_{e_1}G'_1\oplus_{e_2}G'_2\oplus\cdots\oplus_{e_k}G'_k).
\end{aligned}\]
This completes the proof.
\end{proof}

\begin{figure}[h]
\begin{center}
\begin{tikzpicture}
\node[label={above:$1$}] (1) at (0,0) {};
\node[label={below:$2$}] (2) at (240:1) {};
\node[label={below:$3$},xshift=1cm] (3) at (240:1) {};
\node[label={above:}] (4) at (1,0) {};
\node[label={above:}] (6) at (2,0) {};
\node[label={above:}] (11) at (-1,0) {};
\node[label={above:}] (13) at (-2,0) {};
\node[label={above:}] (15) at (-3,0) {};
\node[label={below:},xshift=2cm] (5) at (240:1) {};
\node[label={below:},xshift=3cm] (7) at (240:1) {};
\node[label={below:},xshift=-1cm] (12) at (240:1) {};
\node[label={below:},xshift=-2cm] (14) at (240:1) {};
\node[label={below:}] (9) at (300:2) {};
\node[label={below:},xshift=-1cm] (8) at (300:2) {};
\node[label={below:},xshift=-1cm] (10) at (300:3) {};
\draw (1) -- (2) -- (3) -- (5) -- (7) -- (6) -- (4) -- (1);
\draw (1) -- (3) -- (4) -- (5) -- (6);
\draw (1) -- (11) -- (13) -- (15) -- (14) -- (12) -- (2) -- (11) -- (12) -- (13) -- (14);
\draw (2) -- (8) -- (10) -- (9) -- (3) -- (8) -- (9);
\node[rectangle,draw=none] at (0,-3) {$G_1$};
\end{tikzpicture}
\hfil
\begin{tikzpicture}
\node[label={above:$1$}] (1) at (0,0) {};
\node[label={below:$2$}] (2) at (240:1) {};
\node[label={below:$3$},xshift=1cm] (3) at (240:1) {};
\node[label={above:}] (11) at (-1,0) {};
\node[label={above:}] (13) at (-2,0) {};
\node[label={below:},xshift=-1cm] (15) at (240:2) {};
\node[label={below:},xshift=-1cm] (12) at (240:1) {};
\node[label={below:},xshift=-2cm] (14) at (240:1) {};
\node[label={below:}] (9) at (300:2) {};
\node[label={below:},xshift=-1cm] (8) at (300:2) {};
\node[label={below:},xshift=1cm] (10) at (300:1) {};
\foreach \v/\ang in {4/0,5/60,6/120,7/150}{
\node[label={\ang:}] (\v) at (\ang:1) {};
\draw (1) -- (\v);
}
\draw (1) -- (11) -- (13) -- (14) -- (15) -- (12) -- (2) -- (8) -- (9) -- (10) -- (3) -- (4) -- (5) -- (6) -- (7);
\draw (13) -- (12) -- (14);
\draw (2) -- (3) -- (1) -- (2) -- (11) -- (12);
\draw (8) -- (3) -- (9);
\node[rectangle,draw=none] at (0,-3) {$G_2$};
\end{tikzpicture}
\end{center}
\caption{Graphs obtained from $K_3$ by attaching three linear $2$-tree}
\label{superG1G2}
\end{figure}
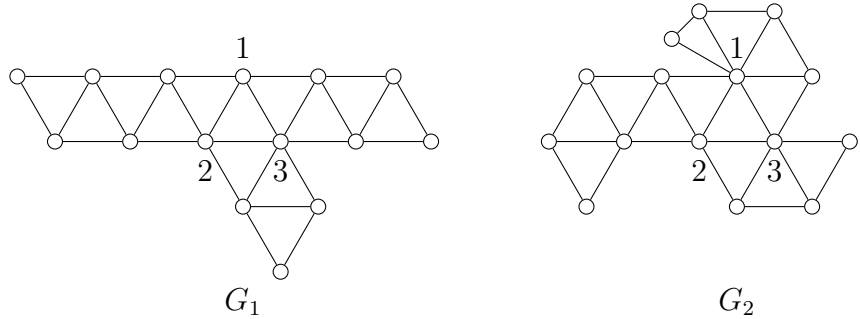

\begin{exam}
Both $G_1$ and $G_2$ in Figure~\ref{superG1G2} are obtained from $K_3$ induced on vertices $\{1,2,3\}$ by attaching three linear $2$-trees.  By Corollary~\ref{cor:attach}, $\detD(G_1)=\detD(G_2)$.
\end{exam}

\section{The $2$-clique paths}
\label{sec:twoCP}

Consider the set $[2,b]$ as an increasing sequence $2,\ldots, b$.  A special family of non-leaping sequences are of the form
\[0,1,[2,p_1-1],[2,p_2-1],\ldots,[2,p_m-1],\]
where $p_1,\ldots,p_m$ are integers at least $3$.
Such a sequence is abbreviated as $2:p_1,\ldots,p_m$.  It is possible that $m=0$, in which case the sequence is $0,1$ and $\CP{0,1}=\{K_2\}$.

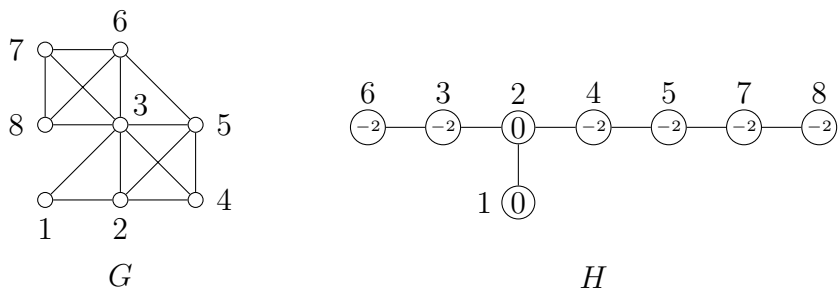
\begin{figure}[h]
\begin{center}
\begin{tikzpicture}
\node[label={below:$1$}] (1) at (0,0) {};
\node[label={below:$2$}] (2) at (1,0) {};
\node[label={45:$3$}] (3) at (1,1) {};
\node[label={right:$4$}] (4) at (2,0) {};
\node[label={right:$5$}] (5) at (2,1) {};
\node[label={above:$6$}] (6) at (1,2) {};
\node[label={left:$7$}] (7) at (0,2) {};
\node[label={left:$8$}] (8) at (0,1) {};
\draw (1) -- (2) -- (3) -- (1);
\draw (4) -- (2) -- (5) -- (3) -- (4) -- (5) -- (6) -- (3);
\draw (7) -- (6) -- (8) -- (3) -- (7) -- (8);
\node[rectangle,draw=none] at (1,-1) {$G$};
\end{tikzpicture}
\hfil
\begin{tikzpicture}[every node/.append style={inner sep=1pt}]
\node[label={left:$1$}] (1) at (0,0) {$0$};
\node[label={above:$2$}] (2) at (0,1) {$0$};
\node[label={above:$3$}] (3) at (-1,1) {\tiny $-2$};
\node[label={above:$4$}] (4) at (1,1) {\tiny $-2$};
\node[label={above:$5$}] (5) at (2,1) {\tiny $-2$};
\node[label={above:$6$}] (6) at (-2,1) {\tiny $-2$};
\node[label={above:$7$}] (7) at (3,1) {\tiny $-2$};
\node[label={above:$8$}] (8) at (4,1) {\tiny $-2$};
\draw (6) -- (3) -- (2) -- (4) -- (5) -- (7) -- (8);
\draw (1) -- (2);
\node[rectangle,draw=none] at (1,-1) {$H$};
\end{tikzpicture}
\end{center}
\caption{A graph $G$ in $\CP{2:3,4,3,4}$ and its reduced graph $H$}
\label{seesawex}
\end{figure}

\begin{exam}
Let $p_1,p_2,p_3,p_4=3,4,3,4$.  The sequence $s=2:3,4,3,4$ stands for
\[0,1,2,2,3,2,2,3.\]
Let this sequence be $q_1,\ldots,q_{8}$.  Figure~\ref{seesawex} shows a graph $G$ from $\CP{s}$ and its reduced subgraph.  The graph is obtained from a disjoint union of four cliques $K_{p_1}$, $K_{p_2}$, $K_{p_3}$, and $K_{p_4}$ by gluing an edge from each of two consecutive cliques.

By Remark~\ref{rem:reduced}, the reduced graph is a weighted graph on $8$ vertices such that
\begin{itemize}
\item vertices $1$ and $2$ are adjacent and with weight $0$;
\item for $k=3,\ldots,8$, vertex $k$ is adjacent to $k-1$ if $q_k=q_{k-1}+1$, and is adjacent to $b_{k-1}$ if $q_k=2$;
\item each edge has weight $1$, and every vertex except for $1$ and $2$ has weight $-2$.
\end{itemize}
Therefore, on the reduced graph $H$, vertices $3$ and $6$ form a path of length $(p_1-2)+(p_3-2)$, and vertices $4,5,7,8$ form another path of length $(p_2-2)+(p_4-2)$.
\end{exam}

Any graph from $\CP{2:p_1,\ldots,p_m}$ is obtained from a disjoint union of $K_{p_1}$, $K_{p_2}$, $\ldots$, $K_{p_m}$ by gluing an edge of $K_{p_i}$ to an edge of $K_{p_{i+1}}$ while an edge cannot be glued twice.  Graphs in $\CP{2:p_1,\ldots,p_m}$ are called \emph{$2$-clique paths}.  When $p_1=p_2=\cdots =p_m=3$, it is the family of linear $2$-trees on $m+2$ vertices.

\begin{defn}
Define $P^{-2}_n$ as a weighted path on $n$ vertices such that each vertex has weight $-2$ and each edge has weight $1$.  The \emph{seesaw graph} $S_{\ell,r}$ is a weighted graph on $n=2+\ell+r$ vertices constructed by the following process:
\begin{itemize}
\item Start with two adjacent vertices $1$ and $2$, each with weight $0$;
\item join an endpoint of $P^{-2}_\ell$ to vertex $2$ by an edge;
\item join an endpoint of $P^{-2}_r$ to vertex $2$ by an edge; and
\item every edge has weight $1$.
\end{itemize}
\end{defn}

\begin{obs}
\label{obs:seesaw}
Let $2:p_1,\ldots,p_m$ be a non-leaping sequence.  Let
\[\ell=\sum_{k\text{ odd}}(p_k-2) \text{ and }r=\sum_{k\text{ even}}(p_k-2).\]
Then the reduced graph of $2:p_1,\ldots,p_m$ is isomorphic to the seesaw graph $S_{\ell,r}$.
\end{obs}

\begin{lem}
\label{wpath}
Let $A(P^{-2}_n)$ be the weighted adjacency matrix of the $P^{-2}_n$.  Then $\det(A(P^{-2}_n))=(-1)^n(n+1)$ and $\inertia(A(P^{-2}_n))=(0,n,0)$.
\end{lem}
\begin{proof}
It is easy to check that
\[\begin{aligned}
\det\begin{bmatrix}-2\end{bmatrix} &= -2=(-1)^1(1+1) \text{ and}\\
\det\begin{bmatrix}-2&1\\1&-2\end{bmatrix} &= 3=(-1)^2(2+1),
\end{aligned}\]
so the statement is true for $n=1,2$.

Assuming the statement is true for small $n$, we will prove by induction.
Suppose the vertices of $P_n^{-2}$ are labeled by $1,2,\ldots, n$ by the path order.  Then the first row of $A(P_n^{-2})$ has only two nonzero entries, namely, $[A(P_n^{-2})]_{1,1}=-2$ and $[A(P_n^{-2})]_{1,2}=1$.  By Laplace expansion,
\[\begin{aligned}
\det(A(P^{-2}_n)) &= -2\det(A(P^{-2}_{n-1}))-\det(A(P^{-2}_{n-2})) \\
 &= -2(-1)^{n-1}(n)-(-1)^{n-2}(n-1) \\
 &= (-1)^{n-2}(2n-n+1)=(-1)^n(n+1).
\end{aligned}\]

By the Gershgorin circles of $A(P^{-2}_n)$, the matrix does not have any positive eigenvalues.  Since the determinant is not zero, $A(P^{-2}_n)$ is a negative definite matrix.
\end{proof}

\begin{thm}
\label{thm:2cpdetinertia}
Let $G\in\CP{2:p_1,\ldots,p_m}$ and $n=|V(G)|$.  Then
\[\detD(G)=(-1)^{n-1}\left(1+\sum_{k\text{ odd}}(p_k-2)\right)\left(1+\sum_{k\text{ even}}(p_k-2)\right)\]
and $\inertiaD(G)=(1,n-1,0)$.
\end{thm}
\begin{proof}
Let $H$ be the reduced graph of $2:p_1,\ldots,p_m$ and $A$ its weighted adjacency matrix.
By Corollary~\ref{cor:detinertia}, $\detD(G)=\det(A)$ and $\inertiaD(G)=\inertia(A)$ for any $G\in\CP{2:p_1,\ldots,p_m}$, so it is enough to find the determinant and the inertia of $A$.

Let
\[\ell=\sum_{k\text{ odd}}(p_k-2) \text{ and }r=\sum_{k\text{ even}}(p_k-2).\]
By Observation~\ref{obs:seesaw}, the reduced graph $H$ is isomorphic to $S_{\ell,r}$.  Thus, up to permutation similarity, we may write $A$ as
\[\begin{bmatrix}
0 & 1 & 0 & \bzero\trans & 0 & \bzero\trans \\
1 & 0 & 1 & \bzero\trans & 1 & \bzero\trans \\
0 & 1 & & & & \\
\bzero & \bzero & \multicolumn{2}{c}{\floating{$A(P^{-2}_{\ell})$}} & \multicolumn{2}{c}{\floating{$O$}} \\
0 & 1 & & & & \\
\bzero & \bzero & \multicolumn{2}{c}{\floating{$O$}} & \multicolumn{2}{c}{\floating{$A(P^{-2}_{r})$}} \\
\end{bmatrix}.\]
By using the first row and column to eliminate the ones on the second row and column, there is a matrix $M$ with $\det(M)=\pm 1$ such that
\[M\trans AM=\begin{bmatrix}0&1\\1&0\end{bmatrix}\oplus
A(P^{-2}_{\ell})\oplus A(P^{-2}_{r}).\]
By Lemma~\ref{wpath},
\[
\begin{aligned}
\detD(G)=\det(A) &= \det \begin{bmatrix}0&1\\1&0\end{bmatrix}\cdot
\det(A(P^{-2}_{\ell}))\cdot \det(A(P^{-2}_{r})) \\
 &= (-1)(-1)^\ell(1+\ell)(-1)^r(1+r) \\
 &= (-1)^{1+\ell+r}(1+\ell)(1+r).
\end{aligned}\]
Since $n=2+\ell+r$, it follows that $\detD(G)=(-1)^{n-1}(1+\ell)(1+r)$.  Also, since $\inertia\begin{bmatrix}0&1\\1&0\end{bmatrix}=(1,1,0)$ and both $A(P^{-2}_\ell)$ and $A(P^{-2}_r)$ are negative definite, $\inertiaD(G)=\inertia(A)=(1,n-1,0)$.
\end{proof}

\begin{thm}
\label{thm:2cpcof}
Let $G\in\CP{2:p_1,\ldots,p_m}$ and $n=|V(G)|$.  Then
\[\cofD(G)=(-1)^{n-1}n.\]
\end{thm}
\begin{proof}
Let $H$ be the reduced graph of $2:p_1,\ldots,p_m$ and $A$ its weighted adjacency matrix.  By Corollary~\ref{cor:cof}, \[\cofD(G)=\det(A+J_{2,n})-\det(A)\]
for any $G\in\CP{2:p_1,\ldots,p_m}$.

Let
\[\ell=\sum_{k\text{ odd}}(p_k-2) \text{ and }r=\sum_{k\text{ even}}(p_k-2).\]
By Observation~\ref{obs:seesaw}, up to permutation similarity, the matrix $A+J_{2,n}$ can be written as
\[\begin{bmatrix}
1 & 2 & 0 & \bzero\trans & 0 & \bzero\trans \\
2 & 1 & 1 & \bzero\trans & 1 & \bzero\trans \\
0 & 1 & & & & \\
\bzero & \bzero & \multicolumn{2}{c}{\floating{$A(P^{-2}_{\ell})$}} & \multicolumn{2}{c}{\floating{$O$}} \\
0 & 1 & & & & \\
\bzero & \bzero & \multicolumn{2}{c}{\floating{$O$}} & \multicolumn{2}{c}{\floating{$A(P^{-2}_{r})$}} \\
\end{bmatrix},
\text{ which leads to }
\begin{bmatrix}
1 & 0 & 0 & \bzero\trans & 0 & \bzero\trans \\
0 & -3 & 1 & \bzero\trans & 1 & \bzero\trans \\
0 & 1 & & & & \\
\bzero & \bzero & \multicolumn{2}{c}{\floating{$A(P^{-2}_{\ell})$}} & \multicolumn{2}{c}{\floating{$O$}} \\
0 & 1 & & & & \\
\bzero & \bzero & \multicolumn{2}{c}{\floating{$O$}} & \multicolumn{2}{c}{\floating{$A(P^{-2}_{r})$}} \\
\end{bmatrix}\]
by applying the Schur complement to the $1,1$-entry; that is, subtract twice of the first column from the second column and then do the same for rows.  Since
\[\begin{bmatrix} -3 &  1 & \bzero\trans & 1 & \bzero\trans\end{bmatrix}
=\begin{bmatrix} -2 & 1 & \bzero\trans & 1 & \bzero\trans\end{bmatrix}
+\begin{bmatrix} -1 & 0 & \bzero\trans & 0 & \bzero\trans\end{bmatrix},\]
it follows that
\[\begin{aligned}
\det(A+J_{2,n}) &= \det(A(P^{-2}_{\ell+r+1}))-\det(A(P^{-2}_{\ell}))\det(A(P^{-2}_{r})) \\
 &=\det(A(P^{-2}_{n-1}))+\detD(G) \\
 &=(-1)^{n-1}n+\detD(G).
\end{aligned}\]
This means $\cofD(G)=\det(A+J_{2,n})-\detD(G)=(-1)^{n-1}n$.
\end{proof}

\begin{rem}
Let $s=2:p_1,\ldots, p_m$.  According to Theorem~\ref{thm:2cpdetinertia}, the distance determinant of a graph in $\CP{s}$ only depends on the sum of the odd terms and the sum of the even terms.  Therefore, applying any permutation to the odd terms and any permutation to the even terms of $s$ will not change the distance determinant.
\end{rem}

\begin{cor}
\label{cor:lineartwotree}
Let $G$ be a linear $2$-tree on $n$ vertices.  Then
\[\detD(G)=(-1)^{n-1}\left(1+\left\lfloor\frac{n-2}{2}\right\rfloor\right)\left(1+\left\lceil\frac{n-2}{2}\right\rceil\right),\]
$\inertiaD(G)=(1,n-1,0)$, and $\cofD(G)=(-1)^{n-1}n$.
\end{cor}
\begin{proof}
The family of linear $2$-tree on $n$ vertices are the graphs in $\CP{2:p_1,\ldots,p_{n-2}}$, where $p_1=\cdots =p_{n-2}=2$.  Then the results follow from Theorems~\ref{thm:2cpdetinertia} and \ref{thm:2cpcof}.
\end{proof}

\section{Applications to the addressing problem}
\label{sec:address}

In this section we will study the applications of our results to the addressing problem and show that $N(G)=|V(G)|-1$ for graphs each of whose blocks are $2$-clique paths.

The following lemma is from \cite{GHH77}.
\begin{lem}
\label{blocklem}
{\rm \cite{GHH77}} If $G$ is a connected graph with blocks $G_1,\ldots,G_r$, then
\[
\begin{aligned}
\cofD(G) &= \prod_{i=1}^r\cofD(G_i) \\
\detD(G) &= \sum_{i=1}^r\detD(G_i)\prod_{j\ne i}\cofD(G_j).
\end{aligned}
\]
\end{lem}

For any real number $x$, define $\sign(x)$ as $1$, $0$, or $-1$ when $x$ is positive, zero, or negative, respectively. Then we have the following lemma about $\sign(\detD(G))$.

\begin{lem}
\label{signlem}
Let $G$ be a connected graph of order $n$ with $r$ blocks $G_1,\ldots,G_r$ of order $n_1,\ldots,n_r$, respectively.  If
\[\sign(\detD(G_i))=\sign(\cofD(G_i))=(-1)^{n_i-1}\]
for $1\leq i\leq r$, then
$$\sign(\detD(G))=\sign(\cofD(G))=(-1)^{n-1}.$$
\end{lem}
\begin{proof}
By Lemma~\ref{blocklem},
\[\begin{aligned}\sign(\cofD(G)) &= \prod_{i=1}^r\sign(\cofD(G_i)) \\
 &= (-1)^{\sum_{i=1}^r(n_i-1)}=(-1)^{n-1}.
\end{aligned}\]
Similarly,
\[\sign(\detD(G_i)\prod_{j\ne i}\cofD(G_j)) = (-1)^{\sum_{i=1}^r(n_i-1)}=(-1)^{n-1}\]
and $\sign(\detD(G))=(-1)^{n-1}$.
\end{proof}

\begin{cor}
\label{signcor}
If $G$ is a connected graph of order $n$ whose blocks are $2$-clique paths, then
$\sign(\detD(G))=\sign(\cofD(G))=(-1)^{n-1}.$
\end{cor}
\begin{proof}
Let the blocks of $G$ be $G_1,\ldots,G_r$ of order $n_1,\ldots,n_r$, respectively. By Theorems~\ref{thm:2cpdetinertia} and \ref{thm:2cpcof},
\[\sign(\detD(G_i))=\sign(\cofD(G_i))=(-1)^{n_i-1}\]
for $1\leq i\leq r$. By Lemma \ref{signlem}, $\sign(\detD(G))=\sign(\cofD(G))=(-1)^{n-1}$.
\end{proof}

In 1950, Jones \cite{Jones50} gave an approach to get the inertia of a symmetric matrix from its principal leading minors. Lemma~\ref{inertialem} states Theorem 4 in \cite{Jones50}.

\begin{lem}
\label{inertialem}
{\rm \cite{Jones50}}
Let $A$ be a nonsingular symmetric $n\times n$ matrix with principal leading minors $D_1,\ldots,D_n$. If there is no consecutive two zeros in the sequence $D_1,\ldots,D_n$, then
$n_-$ is the number of sign changes in the sequence $1,D_1,\ldots,D_n$, ignoring the zeros in the sequence.
\end{lem}

\begin{rem}
The original statement of \cite[Theorem 4]{Jones50} says that $n_-$ is the number of sign changes in the sequence $1,D_1,\ldots,D_n$, where any zero $D_i$ may be given arbitrary sign. It was also shown that every zero is guaranteed to appear between a `$+$' and a `$-$' under the assumption; therefore, ignoring the zeros leads to the same number of sign changes.
\end{rem}

\begin{thm}\label{thm:inertiablock}
If $G$ is a connected graph of order $n$ whose blocks are $2$-clique paths, then
\[\inertiaD(G)=(1,n-1,0).\]
\end{thm}
\begin{proof}
We will show that there is an ordering of the vertices $v_1,\ldots,v_n$ such that the induced subgraph $G[\{v_1,\ldots,v_k\}]$ is also a connected graph whose blocks are $2$-clique paths for any $k\geq 2$.

First we claim that if $H$ is a $2$-clique path with $|V(H)|\geq 3$, then there exists at least two vertices $u$ and $v$ such that both $H-u$ and $H-v$ are a $2$-clique path.  Let $H\in \CP{2:p_1,\ldots,p_m}$.  Then $H$ is obtained by gluing the cliques $K_{p_1},\ldots,K_{p_m}$ into a path-like structure.  If $m\geq 2$, then we may pick a vertex from each of the two ending cliques so that the vertex was not used for gluing.  If $m=1$, then $H$ is a complete graph with $|V(H)|\geq 3$, so removing any vertex from $H$ gives a smaller complete graph.  In either cases, we found two vertices with the desired property.

If $G$ itself is a block with $|V(G)|\geq 3$, then by the previous claim we can find a vertex $v_n$ such that $G-v_n$ is still a $2$-clique paths.  If $G$ has more than one blocks, then there is a pendent block $B$, which is incident to only one cut-vertex $x$.  If $B=K_2$, then let $v_n$ be the vertex in $B$ other than $x$.  If $B$ has more than three vertices, then there are two vertices $u$ and $v$ such that $B-u$ and $B-v$ are still $2$-clique paths, so we may pick $v_n$ as one of $u$ and $v$ that is different from $x$.  In either cases, $G-v_n$ is still a graph whose blocks are $2$-clique paths.  Inductively, keep removing a vertex while preserving the structure, and name the removed vertices as $v_n,\ldots,v_3$.  This process will stop when the remaining graph is $K_2$.  At this point, name the remaining two vertices as $v_2$ and $v_1$ in any order.

Thus, we find an ordering of the vertices $v_1,\ldots, v_n$ such that the induced subgraph $G[\{v_1,\ldots,v_k\}]$ is also a connected graph whose blocks are $2$-clique paths for any $k\geq 2$.  Note that by the way we choose $v_n$, removing $v_n$ does not change the distance of any pair of vertices in the remaining graph.  Therefore,
\[\calD(G[\{v_1,\ldots,v_k\}])=\calD(G)[\{v_1,\ldots,v_k\}],\]
where on the right hand side is the principal submatrix of $\calD(G)$ induced on $\{v_1,\ldots,v_k\}$.  Let $D_k=\det(\calD(G[\{v_1,\ldots,v_k\}]))$.  Then $D_1,\ldots,D_n$ are the principal leading minors by the order $v_1,\ldots, v_n$.  By Corollary~\ref{signcor},
\[\sign(D_k)=\begin{cases}
0 & \text{if }k=1, \\
(-1)^{k-1} & \text{if }k\geq 2.
\end{cases}\]
There are $n-1$ sign changes in the sequence $1,D_1,\ldots,D_n$, so $\inertiaD(G)=(1,n-1,0)$ by Lemma~\ref{inertialem}.
\end{proof}

\begin{cor}
[touch the addressing bound]
\label{cor:touch}
If $G$ is a connected graph of order $n$ whose blocks are 2-clique paths, then
$N(G)=n-1.$
\end{cor}

Note that the set of graphs considered in Theorem~\ref{thm:inertiablock} or Corollary~\ref{cor:touch} contains complete graphs, trees, block graphs, and $2$-clique paths, so they generalize several known results.

\newcommand{\noopsort}[1]{}

\end{document}